\documentclass[a4paper,11pt]{amsart}
\usepackage{graphicx,amssymb}
\usepackage{amsmath,amscd}
\usepackage[english]{babel}
\usepackage[applemac]{inputenc}
\usepackage[small,nohug,heads=littlevee]{diagrams}
\diagramstyle[labelstyle=\scriptstyle]

\newtheorem{thm}{Theorem}[section]
\newtheorem{cor}[thm]{Corollary}
\newtheorem{lem}[thm]{Lemma}
\newtheorem{prop}[thm]{Proposition}
\newtheorem*{thm A}{Theorem A}
\newtheorem*{thm B}{Theorem B}
\theoremstyle{definition}

\theoremstyle{remark}
\newtheorem{rem}[thm]{Remark}

\newcounter{numl}

\newcommand{\labelnuml}{\textup{(\roman{numl})}}
\newenvironment{numlist}{\begin{list}{\labelnuml}%
{\usecounter{numl}\setlength{\leftmargin}{0pt}%
\setlength{\itemindent}{2\parindent}%
\setlength{\itemsep}{\smallskipamount}\def
\makelabel ##1{\hss \llap {\upshape ##1}}}}{\end{list}}

\def\cal{\mathcal}
\def\bb{\mathbb} 

\def\a{\alpha } 
\def\b{\beta }

\def\O{\Omega }

\def\dim{\rm dim\, }

\def\ot{\otimes }
\def\part{\partial }
\def\bpart{\bar\partial }

\def\w{\wedge }

\begin{document}


\title[Lee classes of complex surfaces]{On the Lee classes of locally conformally symplectic complex surfaces}

\author[V. Apostolov]{Vestislav Apostolov} \address{Vestislav Apostolov \\
D{\'e}partement de Math{\'e}matiques\\ UQAM\\ C.P. 8888 \\ Succ. Centre-ville
\\ Montr{\'e}al (Qu{\'e}bec) \\ H3C 3P8 \\ Canada}
\email{apostolov.vestislav@uqam.ca}

\author[G. Dloussky]{Georges Dloussky}\address{Georges Dloussky, Aix-Marseille University, CNRS, Centrale Marseille, I2M, Marseille, France}
\email{georges.dloussky@univ-amu.fr}

\thanks{V.A. was supported in part by an NSERC discovery grant and is grateful to the support of Universit\'e Aix-Marseille  and the Institute of Mathematics and Informatics of the Bulgarian Academy of Sciences where a part of this project was realized. G.D. is grateful to UQAM for the hospitality during the preparation of the work and to CIRGET and UMI  3457 du CNRS for the financial support.   The authors  would also like to thank Steven Boyer, Baptiste Chantraine, Pierre Derbez  for insightful discussions. They are especially grateful to Matei Toma for pointing out  to them the works \cite{CT,brunella-inoue}, and to Max Pontecorvo for clarifying the material in Section~\ref{s:enoki}.}

\date{\today}

\begin{abstract} 
We prove that the deRham cohomology classes of Lee forms of locally conformally symplectic structures  taming the complex structure of a compact complex surface $S$ with first Betti number equal to $1$ is either a non-empty open subset of $H^1_{dR}(S, \bb R)$, or  a single point. In the latter case, we show that $S$ must be biholomorphic to a blow-up of an Inoue--Bombieri surface. Similarly, the deRham cohomology classes of  Lee forms of locally conformally K\"ahler structures of a compact complex surface $S$ with first Betti number equal to $1$ is either a non-empty open subset of $H^1_{dR}(S, \bb R)$, a single point or the empty set. We give a characterization of Enoki surfaces in terms of the existence of a special foliation, and obtain a vanishing result for the Lichnerowicz--Novikov  cohomology groups on the class ${\rm VII}$ compact complex surfaces with infinite cyclic fundamental group.
\end{abstract}

\maketitle


\section{Introduction}

This paper is a sequel to our previous work~\cite{lcs}  in which we have established the following result
\begin{thm}\cite{lcs} \label{thm:lcs}  Any compact complex surface $S=(M,J)$ admits a non-degenerate $2$-form $\omega$ which tames the complex structure $J$, i.e. its $(1,1)$-part $\omega^{1,1}$  with respect to $J$ is positive-definite, and satisfies
\begin{equation}\label{lcs}
d\omega= \a \wedge \omega
\end{equation}
for some closed $1$-form $\a$.
\end{thm}
A non-degenerate $2$-forms $\omega$ satisfying \eqref{lcs}  is called a {\it locally conformally symplectic} (LCS) structure and the corresponding closed $1$-form $\a$ is referred to as the {\it Lee form} of $\omega$.  The notion of an LCS structure is conformally invariant in the sense that if $\omega$ is an LCS structure with Lee form $\a$, then $\tilde \omega=e^f \omega$ is an LCS structure with Lee form $\tilde \a = \a + df$. Thus, the deRham class $[\a]\in H^1_{dR}(S, \bb R)$ is a natural invariant of (the conformal class of) an LCS structure,  which we shall call the {\it Lee class} of $\omega$. We thus  consider the set of all Lee classes of LCS structures taming the complex structure of $S$
$$\mathcal{T}(S): =\{[\a] \in H^1_{dR}(S, \bb R) : \exists \ \omega \in \mathcal{E}^2(S, \bb R) \ {\rm s. \ t.} \ \omega^{1,1}>0, \ d\omega= \a \wedge \omega\}.$$
Theorem~\ref{thm:lcs} then states that $\mathcal{T}(S)\neq \emptyset$. This result is new only when the first Betti number $b_1(S)$ is odd, in which case $0 \notin \mathcal{T}(S)$ (see e.g. \cite[Prop.~3.5]{lcs}). Conversely, when $b_1(S)$ is even, it follows from \cite[Lemme~II.3]{gauduchon} and  \cite[p.~185]{HL} that $\mathcal{T}(S)=\{0\}$.

\smallskip
A further motivation for studying LCS structures comes from the following recent result by Eliashberg--Murphy~\cite{EM} (see also \cite[Thm. 2.15 \& Rem. 2.16]{CM}):
\begin{thm}\label{thm:EM}\cite{EM} Let $(M,J)$ be a compact almost complex manifold and $a\neq 0 \in H^{1}_{dR}(M, \bb R)$ a non-trivial deRham class. Then  for any $C> 1$ sufficiently large  and  any closed $1$-form $\a \in C a$,  there exists an LCS structure with Lee form $\a$,  compatible with the orientation on $M$ induced by $J$.
\end{thm}
The LCS forms $\omega$ constructed in \cite{EM} are in fact {\it exact} in the sense that $$\omega = d_{\a} \beta := d\beta - \a\wedge \beta$$ for a $1$-form $\beta$ on $M$. On a compact complex surface containing a rational curve, such LCS forms cannot tame the underlying complex structure, i.e. generically they  are different from the LCS structures provided by Theorems~\ref{thm:lcs}.

\bigskip
The main result of this paper is the following structure theorem for the set $\mathcal{T}(S)$.
\begin{thm}\label{thm:A} Let $S$ be a compact complex surface with first Betti number equal to $1$. Then, either $\mathcal{T}(S)$ is a non-empty open subset in $H^1_{dR}(S, \bb R)$, or else $\mathcal{T}(S)$ is a single point and $S$ is a blow-up of an Inoue--Bombieri surface~\cite{inoue}.
\end{thm} 
The proof of this result relies, at one hand,  on a characterization of the case when $\mathcal{T}(S)$ is not open in terms of the ``K\"ahler rank'' theory developed in \cite{CT} (see Theorem~\ref{thm:main}, Remark~\ref{kahler-rank} and Lemma~\ref{l:choise-toma} below),  and, at the other hand, on the characterization of the blow-ups of Inoue--Bombieri surfaces obtained by Brunella in \cite{brunella-inoue}.

As observed recently in \cite{otiman},  Theorem~\ref{thm:A} shows that on the Inoue--Bombieri complex surfaces (and their blow-ups) the existence result for LCS structures provided by Theorem~\ref{thm:lcs}  is complementary to the one provided by  Theorem~\ref{thm:EM}, see Corollary~\ref{c:exact} below.

\smallskip
Theorem~\ref{thm:A} is to be compared with recent results of R. Goto~\cite{G}  about  the deformations of Lee classes of locally conformally K\"ahler structures. As a matter of fact, combining \cite[Thm. 2.3]{G} with Theorem~\ref{thm:main} in this paper, we obtain the following
\begin{thm}\label{thm:B} Let $S$ be a compact complex surface with first Betti number equal to $1$ and  $\mathcal{C}(S) \subset \mathcal{T}(S)$ the set of Lee classes of locally conformally K\"ahler structures on $S$, i.e.
$$\mathcal{C}(S)=\{ [\a] \in H^1_{dR}(S, \bb R) : \exists \  \omega \in \mathcal{E}^{1,1}(S, \bb R), \ \omega>0, \ d\omega =\a \wedge \omega\}.$$
Then  $\mathcal{C}(S)$ is either empty, a single point or a non-empty open subset in $H^1_{dR}(S, \bb R)$. 
\end{thm}
Examples of either type do exist,  due to~\cite{B,tricerri}, see \cite[Thm.~1.4]{lcs}.

\bigskip
The proofs of Theorems~\ref{thm:A} and \ref{thm:B} use a vanishing result for the first cohomology group $H_{d_{L}}^1(S, L)$ associated to the sheaf of parallel sections of the flat real line bundle $L$  corresponding to a deRham class $a\in \mathcal{T}(S)$ (see Theorem~\ref{thm:main}), which holds true for all surfaces with $b_1(S)=1$ except the blow-ups of Inoue--Bombieri surfaces,  according to \cite[Thm.~1]{brunella-inoue}, Theorem~\ref{thm:main},  and Lemma~\ref{l:choise-toma} below. One is thus naturally led to ask whether or not the assumption $a\in \mathcal{T}(S)$ can be removed from this statement.  In the final section of the paper,  we recollect some observations regarding this and some related questions,  and establish the following vanishing result.
\begin{thm} \label{thm:Z-vanishing} Let $S$ be a compact complex surface with first Betti number equal to $1$ and fundamental group isomorphic to $\bb Z$. Then, for any real flat bundle  $L$ associated to a class $a\neq 0\in H^1_{dR}(S, \bb R) \stackrel{\exp}{\cong} H^1(S, \bb R^*_{+})$, 
$${\rm dim}_{\bb R} H_{d_{L}}^k(S, L)=0, \ k \neq 2, \ \ {\rm dim}_{\bb R} H_{d_{L}}^2(S, L)=  {\rm dim}_{\bb R} H^2(S, \bb R).$$
\end{thm}

\section{Preliminaries}\label{s:preliminaries}

Let $X=(M,J)$ be a compact complex manifold of complex dimension $n$, and  $\alpha$ a closed $1$-form  on $X$, representing de Rham class $a=[\alpha] \in H^1_{dR}(X, \bb R)$.  We denote by $L_{\alpha}=X \times \bb R$ the topologically trivial real line bundle 
over $X$,  endowed with  the flat connection $\nabla^{\alpha} s := d s + \alpha\otimes s$,  where $s$ is a smooth  function on $X$, also viewed as a smooth section of $L_{\a}$. Similarly, $\nabla^{\alpha}$ induces a holomorphic structure on the complex bundle  $\mathcal{L}_{\a} := L_{\alpha} \otimes \mathbb{C}$ such that parallel sections are holomorphic.  Writing $\a_{\mid U_i} = df_i$ on an open covering $\mathfrak{U}= (U_i)$ of $X$,  $\{(U_i, e^{-f_i})\}$ defines a parallel (respectively holomorphic) trivialization of $L_{\alpha}$ (resp. of $\cal L_{\a}$) with transition functions $e^{f_i-f_j}$ on $U_i \cap U_j$. With respect to this trivialization, $s_0=(U_i, e^{f_i})$ is a nowhere vanishing smooth section of $L_{\alpha}$. This constructions fits in  into the sequence of natural morphisms
\begin{equation}\label{H1}
H^1_{dR}(X,\bb R)\stackrel{\exp}{\cong}H^1(X, \bb R^*_{+}) {\longrightarrow} H^1(X,\bb C^*) \longrightarrow {\rm Pic}(X), 
\end{equation}
where $\mathbb R^*_{+}$ denotes the sheaf of locally constant positive real functions,  and ${\rm Pic}(X) = H^1(X, \cal O^*)$ is the group of isomorphism classes of holomorphic line bundles.  Indeed,  $L_{\alpha}$ represents  the isomorphism class  $\exp(a)\in H^1(X, \bb R^*_{+})$ given by \eqref{H1}  whereas  $\cal L_{\a}$ represents  its image in ${\rm Pic}^0(X),$ where  ${\rm Pic}^0(X)$ denotes the subgroup of $H^1(X, {\cal O}^*)$ of isomorphism classes of holomorphic line bundle with zero first Chern class. 

In what follows,  we shall tacitly identify $L_{\a}$ and $L_{\tilde \a}$ (resp. $\mathcal L_{\a}$  and $\mathcal{L}_{\tilde \a}$) for any two  $\a, \tilde \a \in a$, and denote (with a slight abuse of notation) by $L_{a}$ (resp. $\cal L_{a}$)  a  flat line bundle  obtained by some choice of $\a \in a$; we shall refer to  $L_{a}$  (resp. $\cal L_a$) as the flat real (resp. the flat holomorphic) line bundle {corresponding} to $a\in H^1_{dR}(X, \bb R)$.   Similarly, we shall implicitly identify a flat real line bundle  (resp. a holomorphic line bundle) with the class it represents in $H^1(X, \bb R^*_{+})$ (resp. in ${\rm Pic}(X)$).


We denote by  $\mathcal{E}^k(X, \bb R)$, resp. $\mathcal{E}^{p,q}(X, \bb C)$ the space of smooth real $k$-forms on $X$,  resp. of  smooth complex-valued $(p,q)$-forms on $X$.
The $\a$-twisted differential 
$$d_{\alpha}: =d-\alpha\wedge \cdot$$ defines the Lichnerowicz--Novikov complex
\begin{equation}\label{lichne-novikov}
\cdots \stackrel{d_{\alpha}}{\to}\cal E^{k-1}(X, \bb R) \stackrel{d_{\alpha}}{\to} \cal E^{k}(X, \bb R)  \stackrel{d_{\alpha}}{\to}\cdots 
\end{equation}
which is  isomorphic to the de Rham complex of differential forms with values in $L^*$
\begin{equation}\label{twisted-deRham}
\cdots \stackrel{d_{L^*}}{\to}\cal E^{k-1}(X,L^*) \stackrel{d_{L^*}}{\to} \cal E^{k}(X, L^*)  \stackrel{d_{L^*}}{\to}\cdots
\end{equation}
associated to the sheaf of locally constant sections of $L^*$. This can be viewed by writing $\alpha_{\mid U_i}=df_i$ on an open covering  $\mathfrak {U}=(U_i)$ of $X$: then, for any  $d_{\alpha}$-closed smooth form $\omega$ on $X$,   ${\omega_i}_{\mid U_i} := e^{-f_i}\omega$ gives rise to a smooth section  in $\mathcal{E}^k(X, L^*),$ satisfying $d\omega_i=0$ on $U_i$, i.e. $\{U_i, \omega_i)$ defines a $d_{L^*}$-closed form with values in $L^*$.  In particular,  we have an isomorphism between the cohomology groups
\begin{equation}\label{deRham-isom}
H^{k}_\a(X, \bb R)\cong H^k_{d_{L^*}}(X,L^*),
\end{equation}
associated to the complexes \eqref{lichne-novikov} and \eqref{twisted-deRham}, respectively.

Considering complex-valued forms, one can similarly introduce the operators 
$$d_{L^*}=\part _{\mathcal {L^*}}+\bpart_{\mathcal {L}^*}, \quad {\rm and}\quad d_{\alpha}=\part_{\alpha}+\bpart_{\alpha}$$
with 
$$\part_{\alpha}=\part-\alpha^{1,0}\w\quad {\rm and}\quad \bpart_{\alpha}=\bpart-\alpha^{0,1}\w,$$
acting respectively on $\cal E^{p,q}(X, \cal L^*)$ and $\cal E^{p,q}(X, \bb C)$. These  give rise to the isomorphisms 
\begin{equation}\label{dolbeault-isom}
H^{p,q}_{\bar \partial_{\alpha}}(X, \bb C)\cong H^{p,q}(X,\cal L^*) \cong H^q(X, \O^p\otimes \cal L^*),
\end{equation}
where  $H^{p,q}_{\bar \partial_{\a}}(X, \bb C)= {\rm Ker}(\bar \partial_{\a})/{\rm Im}(\bar \partial_{\a})$ whereas $H^{p,q}(X,\cal L^*)\cong H^q(X, \O^p\otimes \cal L^*)$ is the usual Dolbeault cohomology group of $X$ with values in the flat holomorphic line bundle ${\mathcal L}^*$, and  $\O^p$ stands for  the holomorphic vector bundle of $(p,0)$ forms on $X$.

For any $(k-1)$-form $\phi$ and $(2n-k)$-form $\psi$ on $X$, we have
\begin{equation*}\label{star}
d (\phi \wedge \psi) = (d_{\alpha} \phi) \wedge \psi  + (-1)^{k-1} \phi \wedge (d_{-\alpha}\psi).
\end{equation*}
Integrating the above formula over the closed manifold $X$ leads to a natural pairing between $H^k_{\a}(X, \bb R)$ and $H^{2n-k}_{-\a}(X, \bb R)$. 

For any Riemannian metric $g$ on $X$, the $L^2$ adjoint of $d_{\a}$ is 
\begin{equation*}\label{*d_a}
d^*_{\a}= -* d_{-\a} *,
\end{equation*}
where $*$ is the Hodge operator with respect to $g$. (We  have used that $X$ is oriented and even dimensional.) It follows that the corresponding twisted Laplace operator $\Delta_{\a}^g= d_{\a} d^*_{\a} + d^*_{\a} d_{\a}$ has the same index as the usual Laplacian $\Delta^g$ and satisfies
$$ * \Delta^g_{\a} = \Delta^g_{-\a} *.$$
Hodge theory (see e.g. \cite{AK,Simpson}) and \eqref{deRham-isom}  then imply  that  the spaces $H^k_{\a}(X, \bb R)$ are finite dimensional and the pairing between $H^k_{\a}(X, \bb R)$ and $H^{2n-k}_{-\a}(X, \bb R)$ defined above is a perfect pairing, i.e. 
\begin{equation}\label{poincare}
H^k_{\a}(X, \bb R) \cong \big(H^{2n-k}_{-\a}(X, \bb R)\big)^*, \ \  H_{d_L}^{k}(X, {L}) \cong \big(H_{d_{L^*}}^{2n-k}(X, {L}^*)\big)^*,
\end{equation}
where the upper  $*$ denotes the dual vector space. The index theorem also implies (as observed in \cite{FP})
\begin{equation}\label{euler}
\sum_{k=0}^{2n} (-1)^k {\rm dim}_{\bb R} \ H^k_{\a}(X, \bb R) = e(X),
\end{equation}
where $e(M)$ is the Euler characteristic of $M$. Notice that if $[\a] \neq 0 \in H^1_{dR}(X, \bb R)$, then 
\begin{equation}\label{vanishing-0}
H^0_{\a}(X, \bb R) = H^{2n}_{\a}(X, \bb R)= \{0\},
\end{equation}
Indeed, suppose $d_{\a} f =0$ for some smooth non-zero function $f$. This means that $f$ satisfies the linear system $df = f \a$, so  $f$ cannot vanish on $X$, showing that $\a= d\log |f|$, a contradiction. Thus,  $H^{0}_{\a}(X, \bb R)= \{0\}$ and $H^{0}_{-\a}(X, \bb R) \cong (H^{2n}_{\a}(X, \bb R))^*=\{0\}$.

We shall use the following elementary fact
\begin{lem}\label{cover} Suppose $M$ is a compact manifold and $p: \tilde M \to M$ a finite cover. For any flat real line bundle $L=L_\a$ denote $\tilde L = L_{\tilde \a}$ the corresponding pullback to $\tilde M$, where $\tilde \a = p^* \a$. Then the natural pull-back map $p^*: H_{\a}^k(M, \bb R) \to H_{\tilde \a}^k(\tilde M, \bb R)$ is injective.
\end{lem}
\begin{proof} We need to show that if $\beta$ is  $d_{\a}$-closed $k$-form $\beta$ on $M$ such that $\tilde \beta :=p^*(\beta)=d_{\tilde \a} \tilde \gamma$  on $\tilde M$, then $\beta$ is $d_\a$-exact on $M$. Denote by  $\Gamma$ the finite group of diffeomorphisms of $\tilde M$ such that $M= \tilde M/\Gamma$.  As both $\tilde \a$  and $\tilde \b$ are  invariant under the action of $\Gamma$ of $\tilde M$, the average of $\tilde \gamma$ over $\Gamma$ is a $\Gamma$-invariant $(k-1)$-form $\gamma$  on $\tilde M$,   satisfying  $\tilde \b = d_{\tilde \a} \gamma$. As $\gamma$ descends to $M$ (being $\Gamma$-invariant), we also have $\beta = d_{\a} \gamma$ on $M$. \end{proof}

\smallskip
Similarly to \eqref{poincare}, we have  isomorphisms
\begin{equation*}\label{serre}
H_{\bar \partial_{\a}}^{p, q}(X, \bb C) \cong (H^{n-p, n-q}_{\bar \partial_{-\a}}(X, \bb C))^*, \ \  H^{p,q}(X, \cal L) \cong \big(H^{n-p, n-q}(X, \cal L^*)\big)^*
\end{equation*}
where the second identification is the usual Serre duality and the first follows from the second via \eqref{dolbeault-isom}.

\bigskip
We shall now specialize to the case when $X=S$ is a compact complex surface. For a flat real line bundle $L:=L_{\a}$ and $\cal L= L_{\a}\otimes \bb C$ the corresponding flat holomorphic line bundle,  we shall denote by
\begin{enumerate}
\item[$\bullet$] $b_k(S, L) := {\rm dim}_{\bb R} \ H^k_{d_{L}}(S,L)={\rm dim}_{\bb R} \ H_{-\a}^k(S, \bb R),$
\item[$\bullet$] $h^{p,q}(S, \cal L) := {\rm dim}_{\bb C} H^{p,q}(S,\cal L)={\rm dim}_{\bb C} H^{p,q}_{\bar \partial_{-\a}}(S, \bb C),$
\end{enumerate}
the corresponding dimensions.  We then have 
\begin{lem}\label{blow-up} Let $S$ be a compact complex surface,  $L=L_{\a}$ a flat real line bundle corresponding to a closed $1$-form $\a$ and $\cal L = L_{\a} \otimes \bb C$ the corresponding flat holomorphic line bundle. Let $B_{x} : \hat S \to S$ be the blow-down map from the complex surface $\hat S$ obtained from $S$ by blowing up a point $x \in S$ and denote by  ${\hat \a} = B_{x}^{*}(\a)$,  $\hat L= L_{\hat \a}$ and $\hat {\cal L} = {\cal L}_{\hat \a}$ the corresponding objects on $\hat S$,  obtained by the natural pull-back map. Then,
\begin{enumerate}
\item[\rm (a)] $b_1(\hat S, \hat L)=b_1(S, L), \ b_3(\hat S, \hat L)= b_{3}(S, L), \ \ b_2(\hat S, \hat L) = b_2(S, L) + 1$;
\item[\rm (b)] $h^{k,0}(\hat S,  \hat {\cal L}) = h^{k, 0}(S, \cal L), k=0, 1,2.$
\end{enumerate}
\end{lem}
\begin{proof} (a)  Notice that as $e(\hat S)= e(S) +1$ (see e.g. \cite{bpv}), the last equality in (a) follows from the first two and \eqref{euler}-\eqref{vanishing-0}. Also, using the duality $H^3_{\a}(S, \bb R) \cong (H^1_{-\a}(S, \bb R))^*$ (see \eqref{poincare}), it is enough to show  that for each $[\a]\in H^1_{dR}(S, \bb R)$, ${\rm dim}_{\bb R} H^1_{\a}(\hat S, \bb R)={\rm dim}_{\bb R} H^1_{\a}(S, \bb R)$. As the dimension of $H^{1}_{\a}(S, \bb R)$ does not depend on the choice of $\a \in a$, we can choose $\a$ such that it identically vanishes  on a open ball $U$ centred at $x$.  We are going to prove that the natural pull-back map $B_{x}^* :  H_{\a} ^1(S, \bb R) \to H^1_{\hat \a}(\hat S, \bb R)$ is then an isomorphism.

We shall first prove that $B_x^*$ is surjective. With our choice for $\a$, any $d_{\hat \a}$-closed $1$-form $\hat \varphi$ on $\hat S$  is closed over $\hat U= B_x^{-1}(U)$.  As $H^{1}_{dR}(\hat U, \bb R) \cong  H^1_{dR}(\mathbb{C} P^{1}, \bb R)=\{0\}$, we can  write ${\hat \varphi}_{\vert \hat U} = d({\hat \xi}_{\vert \hat U})$. Multiplying ${\hat \xi}_{\vert \hat U}$ by the pull-back via $B_x$ of a bump function centred at $x$ and supported in $U$, we can assume $\hat \xi$ is globally defined on $\hat S$ and $\hat \phi = \hat \varphi - d_{\hat \a} \hat \xi$ is another form representing $[\hat \varphi] \in H^1_{\hat \a}(\hat S)$ which vanishes identically on a neighbourhood of $E$.  Then, the diffeomorphism $(B_x^{-1}) : S\setminus \{x\} \to \hat S \setminus E$ allows us to define a smooth $1$-form $\phi = (B_x^{-1})^* (\hat \phi)$ on $S$ with $d_\a \phi =0$ and $B_x^*(\phi)= \hat \phi$. 
 
We now prove that $B_x^*: H^1_{\a}(S, \bb R) \to H^1_{\hat \a}(\hat S, \bb R)$ is injective. Suppose $\varphi$ is a $d_{\a}$-closed $1$-form on $S$,  such that $\hat \varphi =B_x^*(\varphi) = d_{\hat \a} {\hat \xi}$. As $H^1_{dR}(U, \bb R)=\{0\}$, we can modify $\varphi$ with a $d_\a$-exact $1$-form (as we did above with $\hat \varphi$) and assume without loss that $\varphi_{\vert U} \equiv 0$. It follows that the  function $\hat \xi$ satisfies $d{\hat \xi}_{\vert \hat U} \equiv 0$, i.e. $\hat \xi$ is a smooth function on $\hat S$ which is constant on $\hat U$ and, therefore, is the pull back to $\hat S$ of a smooth function $\xi$ on $S$ (which is constant on $U$). It follows that $\varphi = d_{\a} \xi$.

\smallskip
(b) Again, we assume without loss  that the closed $1$-form  $\a$  identically vanishes  on a open ball $U$ centred at $x$. Clearly, we have an injective pull-back  map $B_x^* : H^{k,0}_{\bar \partial_{\a}}(S, \bb C) \to H^{k,0}_{\bar \partial_{\hat \a}}(\hat S, \bb C)$ so we need to establish its surjectivity. Suppose $\hat \beta$ is a $(k,0)$-form on $\hat S$ satisfying  $\bar \partial_{\hat \a} \hat \beta =0$. Pulling back $\hat \beta$  by the biholomorphism $B_x^{-1}  : S\setminus \{x\} \to \hat S \setminus E$  defines a $(k,0)$-form $\beta$ on $S\setminus \{x\}$,  which satisfies $\bar \partial_{\a} \beta =0$. As $\a$ vanishes on $U$, the $(k,0)$-form $\beta$ is holomorphic on $U\setminus \{x\}$, and therefore extends over $x$ by Hartogs' extension theorem. By construction, $\hat \beta = B_x^*(\beta)$ on $\hat S \setminus E$, hence everywhere by continuity. \end{proof}

\section{Complex surfaces with $b_1=1$} 

From now on, $S$ will denote a compact complex surface whose first Betti number $b_1(S)=1$.  Kodaira~\cite{kodaira} has shown that  for such a surface either $H^0(S, K_S^m)=\{0\}$ for all $m\ge 1$, where $K_S= \O^2$ stands for the canonical bundle of $S$, or  there exists $m_0\ge 1$ such that $K^{m_0}_S \cong \cal O$ is trivial. In the first case, the surface is said to belong to  the class ${\rm VII }$ (we follow the terminology of \cite{bpv}) whereas in the latter case, Kodaira proved that the minimal model $S_0$ of $S$ must be a {\it secondary Kodaira surface}, see \cite{kodaira,bpv}.  The classification of compact complex surfaces in the class ${\rm VII}$ is still open, but the special case  when the minimal model $S_0$ of $S$ satisfies $b_2(S_0)=0$ has been settled by \cite{bogomolov, andrei, yau-et-al}: $S_0$ must then be either a Hopf surface~\cite{kato-hopf} or an Inoue--Bombieri surface~\cite{inoue}. The class  of the minimal complex surfaces $S_0 \in {\rm VII}$ for which $b_2(S_0)>0$ is commonly denoted by  ${\rm VII}_0^+$. We summarize the situation in the following
\begin{thm}\label{thm:kodaira}\cite{bogomolov,kodaira,andrei,yau-et-al} Any compact complex surface $S$ with first Betti number $b_1(S)=1$ is obtained by blowing up a minimal complex surface $S_0$ of one of the following types
\begin{enumerate}
\item[$\bullet$] a secondary Kodaira surface;
\item[$\bullet$] a Hopf surface;
\item [$\bullet$] an Inoue--Bombieri surface;
\item[$\bullet$] a minimal complex surface in the class ${\rm VII}_0^+$.
\end{enumerate}
\end{thm}
We shall use the following key vanishing result (see e.g. \cite[Lemma~2.13]{lcs} for a proof).
\begin{lem}\label{l:vanishing}\cite{Nakamura, lcs} Let $S$ be a compact complex surface whose minimal model belongs to the class {\rm VII}$_0^+$. Then, for any non-trivial holomorphic line bundle $\cal L \in {\rm Pic}^0(S)$
$$H^2(S, \cal L) \cong H^0(S, K_S\otimes \cal L^*)=\{0\}.$$
\end{lem}

\section{The space of Lee classes}

\subsection{A characterization  of the case when $\mathcal T(S)$ is a single point} In this section we are going to establish the following  result.

\begin{thm} \label{thm:main} Let $S$ be a compact complex surface with $b_1(S)=1$ and  $\mathcal{T}(S)\subset H^1_{dR}(S, \bb R)$ the set of Lee classes of LCS forms taming the complex structure on $S$. Then the following conditions are equivalent.
\begin{enumerate}
\item $\exists \ 0\neq a\in H^1_{dR}(S, \bb R)$ and $\a \in a$ such that 
$d_{-\a}d^c_{-\a} (1) = dJ\a + \a \wedge J\a=0$;
\item $\mathcal{T}(S)= \{ a \};$
\item $\mathcal{T}(S) \subset H^1_{dR}(S, \bb R)$ is not open;
\item $\exists \ a \in \mathcal{T}(S)$ such that $H_{d_{L_a}}^1(S, L_{a}) \neq \{0\}.$
\end{enumerate}
\end{thm}

We  first need a variation  of a result in \cite{G}:

\begin{prop}\label{prop:open} Let $M$ be a compact  $2n$-dimensional manifold  which admits  an LCS structure with Lee  class $a\neq 0 \in H^1_{dR}(M, \bb R)$. Let $L=L_{a}$ be the corresponding flat real  line bundle and $L^*$ its dual.   If $H_{d_{L^*}}^3(M, L^*) =\{0\}$, then for any $b\in H^1_{dR}(M, \bb R)$ there exists $\varepsilon >0$ such that for $|t|<\varepsilon$,  $M$ admits an LCS structure with Lee class  in $(a+tb) \in  H^1_{dR}(M, \bb R)$. The statement holds true for the  Lee classes of  LCS structures  taming a given  almost complex structure $J$ on $M$. 
\end{prop}
\begin{proof}
Let $\a, \b$ be representatives of the de Rham classes $a$ and $b$,  respectively, and $\omega_0$ a smooth $2$-form on $M$ satisfying $d_{\a} \omega_0 =0$ and  $\omega_0^n \neq 0$. As the latter condition is open, it is enough to construct a $\mathcal{C}^{\infty}$ family $\omega(t)$ of $2$-forms,  satisfying $d_{(\a+t\b)} \omega(t)=0$ for $|t|<\varepsilon$. (The case when $\omega_0$ tames $J$ is handled similarly, by noting that $\omega^{1,1}>0$ is an open condition.) To this end, we take $\omega(t)$ be a formal power series
$$\omega(t) = \sum_{i=0}^{\infty} \omega_i t^i,$$
where $\omega_i$ are smooth $2$-forms on $M$ (and $\omega_0$ is the $2$-form chosen above). Using $d_\a \omega_0 =0$, the condition $d_{(\a + t\b)}\big(\omega(t)\big)=0$ reads as
\begin{equation}\label{induction}
d_\a \omega_{i+1} =\b \wedge \omega_i,  \ \ i= 0, \ldots
\end{equation}
which can be used to build $\omega_i$ by induction: Indeed, the right hand side is $d_{\a}$-closed as 
$$d_{\a}(\beta \wedge \omega_i)=- \beta\wedge (d_{\a} \omega_i)= -\b \wedge (\b\wedge\omega_{i-1})=0.$$
As $H^3_{\a}(M, \bb R) \cong H_{d_{L^*}}^3(M,L^*)=\{0\}$ by the hypothesis,   one can solve \eqref{induction} inductively as follows. Let $g$ be a Riemannian metric on $M$ and  $d^*_{\a}=- * d_{-\a}*$ the formal $L_2$-adjoint of $d_{\a}$  (with respect to $g$) acting on $p$-forms. The vanishing of $H^3_{\a}(M, \bb R)$ ensures that the Laplacian $\Delta^g_{\a}= d_\a d^*_\a + d^*_\a d_\a$ is invertible on $\mathcal{E}^3(M, \bb R)$,  with inverse $\mathbb{G}_{\a}$. Note that, as $\b\wedge \omega_i$ is a $d_\a$-closed $3$-form, $d_\a\Big(\mathbb{G}_{\a}(\b\wedge \omega_i)\Big)=0$. Letting
$$\omega_{i+1} := d_\a^*({\mathbb G}_\a(\b\wedge \omega_i),$$
we have
\begin{equation*}
\begin{split}
d_\a \omega_{i+1} &= \Delta^g_\a \Big({\mathbb G}_\a(\b\wedge \omega_i)\Big)- d^*_\a d_\a\Big({\mathbb G}_\a(\b\wedge \omega_i)\Big)\\
                                  &= \b\wedge \omega_i.
\end{split}
\end{equation*}
The convergence of the series in $\mathcal{C}^{k}(M)$ follows by  standard Schauder estimates for the Laplacian, whereas the smoothness  of the solutions follows from standard regularity theory for elliptic PDE's. Indeed, notice that  $\omega(t)$ satisfies the PDE 
$$(d^*_\a d_{\a+ t\b} + d_{\a + t\b}d^*_\a) (\omega(t)) = d_{\a + t\b}d^*_\a \omega_0,$$
in which the rhs is $\mathcal{C}^{\infty}$. The above PDE is elliptic for $t$ small enough (as it is so for $t=0$).\end{proof}

\smallskip
\noindent
{\it Proof of Theorem~\ref{thm:main}}. `(1) $\Rightarrow$ (2)':  By Theorem~\ref{thm:lcs}, $\mathcal{T}(S)\neq \emptyset$. Let $\tilde a \in \mathcal{T}(S)$.  As $b_1(S)=1$,  $\tilde a$ can be represented by the closed $1$-form $t\alpha$ for some real constant $t\neq 0$. Let $\omega$ be a $d_{t\alpha}$-closed $2$-form with $\omega^{1,1}>0$. Using that $d_{-\a}d^c_{-\a} (1) =0$, we have 
$$d_{-t\alpha} d^c_{-t\alpha} (1) = tdJ\alpha  + t^2 \alpha \wedge J\alpha  =  t(t-1)\a \wedge J\a.$$
It follows that
\begin{equation*}
\begin{split}
0 =  \int_M d^c_{-t\a} (1) \wedge d_{t\a} \omega  &= \int_M d_{-t\a} d^c_{-t\a} (1) \wedge \omega  \\
                                                                          &= t(t-1) \int_M \a\wedge J\a \wedge \omega \\
                                                                          &= t(t-1) \int_M \a\wedge J\a \wedge \omega^{1,1}.
\end{split}
\end{equation*}
As $\int_M \a\wedge J\a \wedge \omega^{1,1}>0$, it follows that $t=1$, i.e. $\mathcal{T}(S)= \{[\a]\}$.
\smallskip

`(2) $\Rightarrow$ (3)'  is obvious.

\smallskip

`(3) $\Rightarrow$ (4)'  follows from Proposition~\ref{prop:open} and the facts that $b_1(S)=1$ and $H_{d_{L^*}}^3(S, L^*) \cong H_{d_L}^1(S, L)$, see \eqref{poincare}.

\smallskip

`(4)  $\Rightarrow$ (1)': We shall consider four cases, according to the type of the minimal model $S_0$ in Theorem~\ref{thm:kodaira}

\smallskip
\noindent {\it Cases 1 and 2}: $S_0$ is either a secondary Kodaira surface or a  Hopf surface. These cases are impossible because of the following 
\begin{lem}\label{l:hopf} If the minimal model of $S$ is a secondary Kodaira surface or a  Hopf surface, then for any $a\neq 0 \in H^1_{dR}(S, \bb R)$, $H^1_{d_{L_{a}}}(S, L_a)=\{0\}$.
\end{lem}
\begin{proof} Belgun~\cite{B} has shown that any secondary Kodaira surface $S_0$ admits a Vaisman locally conformally K\"ahler structure $\omega_0$, i.e. an LCS structure $\omega_0$ with $\omega_0 = \omega_0^{1,1}>0$ and such that the Lee form $\a_0$ is parallel with respect to $g_0(\cdot, \cdot)=\omega_0(\cdot, J\cdot).$  By \cite{leon}, $H^1_{t\a_0}(S_0, \bb R)=\{0\}$ for any $t\neq 0$.  As $b_1(S_0)=1$, it follows that  for any $\check{a} \neq 0 \in H^1_{dR}(S_0, \bb R)$, $H^1_{d_{L_{\check{a}}}}(S_0, L_{\check{a}})=\{0\}$. As the blow-down map  $b: S\to S_0$ induces an  isomorphism between $H^1_{dR}(S_0, \bb R)$  and $H^1_{dR}(S, \bb R)$, any non-zero class $a \in H_{dR}^1(S, \bb R)$ is the pull-back of  a class $\check{a} \neq 0\in H^1_{dR}(S_0, \bb R)$, thus by Lemma~\ref{blow-up},  $H^1_{d_{L_{a}}}(S, L_{a})=\{0\}$. 

Hopf surfaces are classified in \cite{kato-hopf}, and they are either primary  or secondary. Gauduchon--Ornea~\cite{GO}  showed that the primary Hopf surfaces $S_0$ admit Vaisman locally conformally metrics, so that the same argument as in the case of a secondary Kodaira surface proves the claim. If $S_0$ is a secondary Hopf surface, it is covered by a primary one, $\hat S_0$ say.  Using Lemma~\ref{cover} and Lemma~\ref{blow-up},  we conclude again that $H^1_{d_{L_a}}(S, L_a)=\{0\}$ for each non-trivial flat real line bundle $L=L_{a}$.  \end{proof}

\smallskip
\noindent
{\it Case 3}: $S_0$ is an Inoue--Bombieri surface. In this case our  claim   follows from 
\begin{lem}\label{l:inoue} If the minimal model of $S$ is an Inoue--Bombieri surface,  then $\mathcal{T}(S)=\{a\}$, $H^1_{d_{L_a}}(S, L_{a}) \neq 0$ and there exits $\a \in a$ satisfying the identity in Theorem~\ref{thm:main}~{\rm (1)}.
\end{lem}
\begin{proof} The Inoue--Bombieri surfaces $S_0$ are classified in \cite{inoue} and appear as three types of quotients of  $\bb H \times \bb C$,  where $$\bb H=\{w=w_1 + i w_2, : w_2>0\}$$ denotes the upper half-plane.  An inspection  upon the explicit forms of the desk transformations on $\bb H \times \bb C$ (see \cite{inoue}) shows that in each case the $(1, 0)$-form $-i dw/w_2$  is invariant and therefore descends to $S_0$. Let $\alpha:= dw_2/w_2 = {\rm Re}(-idw/w_2)$ denote the corresponding real $1$-form on $S_0$. We thus have that $J\alpha= {\rm Im}(-idw/w_2)= -dw_1/w_2$ satisfies 
\begin{equation}\label{key-identity}
dJ\alpha =  -dw_1\wedge dw_2/w_2^2 = - \alpha \wedge J\alpha.
\end{equation}
Pulling back $\alpha$ from $S_0$  to $S$ by the blow-down map, we have a non-zero $1$-form  $\a$ on $S$ satisfying \eqref{key-identity}.  Note that $a:=[\a] \neq 0$: otherwise, if $\a = df$, contracting \eqref{key-identity} with  the fundamental 2-form $F$ of a Gauduchon metric $g$ implies
$$\Delta^g f   + g(df, \theta^g) = g(df, df) \ge 0,$$ where $\Delta^g$ is the Laplace operator with respect to $g$ and $\theta^g = J \delta^g F$. By the maximum principle, $f$ must be a constant and $\a=df=0$, a contradiction. By `(1) $\Rightarrow$ (2)', we  have $\mathcal{T}(S)=\{[\a]\}$. By Proposition~\ref{prop:open}, $H^3_{\a}(S, \bb R) \neq \{0\}$ so that, by \eqref{poincare} and \eqref{deRham-isom},  ${\rm dim}_{\bb R}\ H^3_{\a}(S, \bb R) = {\rm dim}_{\bb R}\  H^1_{-\a}(S, \bb R) = {\rm dim}_{\bb R}  H^1_{d_{{L_{\a}}}}(S, L_{\a})\neq 0$.  
Alternatively, in order to show $H^1_{d_{{L_{\a}}}}(S, L_{\a}) \neq \{0\}$ one can use a direct computation on $S_0$, as in \cite{otiman}. \end{proof}

\smallskip
\noindent {\it Case 4}:  $S_0$ belongs to  the class ${\rm VII}_0^+$. Let $[\a]  \in \mathcal{T}(S)$  and $0\neq [\b] \in H^1_{-\a}(S, \bb R)$. As $d_{-\a} \b =0$, $\bar \partial_{-\a} \b ^{0,1}=0$ where $\b^{0,1}= \frac{1}{2}(\b - i J\b)$ is the $(0,1)$-part of $\b$. As $H^{0,1}_{\bar \partial_{-\a}}(S, \bb C) \cong H^1(S, \cal L_{\a})$ and $H^2(S, \cal L_{\a})=\{0\}$ (see Lemma~\ref{l:vanishing}), it follows from the Riemann--Roch  formula and  \cite[Lemmas~2.4 \& 2.11]{lcs} that ${\rm dim}_{\bb C} \ H^{0,1}_{\bar \partial_{-\a}}(S, \bb C) = {\rm dim}_{\bb C} \ H^0(S, {\cal L}_{\a})=0$. Thus, $\b^{0,1} = \bar \partial_{-\a} (h+ i f)$ for some smooth real-valued functions $f, h$ on $S$. Equivalently, $\beta = d_{-\a} h + d_{-\a}^c f$.  Since 
$[\b] \neq 0 \in H^1_{-\a}(S, \bb R)$, $f$ is not identically zero whereas $d_{-\a} \b =0$ implies
\begin{equation}\label{PSH}
d_{-\a}d^c_{-\a} f =  dJdf + \a\wedge Jdf - J\a\wedge df + f(dJ\a + \a \wedge J\a)=0.
\end{equation}
Let $\omega$ be a $d_{-\a}$-closed $2$-form with $(1,1)$-part $F=\omega^{1,1} >0$. We consider the hermitian metric $g$ on $S$ whose fundamental $2$-form is $F$. By \cite[Lemmas~2.4 \& 2.5]{lcs}, we have
\begin{equation}\label{l:lcs}
\delta^g(\theta^g-\a) + g(\theta^g-\a, \a)=0,
\end{equation}
where we recall $\theta^g = J \delta^g F$. Taking contraction with $F$ in \eqref{PSH} yields
\begin{equation*}\label{M(f)}
M(f):= \Delta^g f + g(df, \theta^g-2\a) + f\big(\delta^g \a  + g(\a, \theta^g-\a)\big) =0.
\end{equation*}
The adjoint operator $M^*(f)$  of $M(f)$ with respect to the $L^2$-product induced by $g$ is 
\begin{equation*}
\begin{split}
M^*(f) &= \Delta^g f - g(d f, \theta^g-2\alpha) + f\Big(\delta^g(\theta^g-\a) + g(\a, \theta^g-\a)\Big) \\
           &=  \Delta^g f - g(d f, \theta^g-2\alpha),
           \end{split}
           \end{equation*}
where for the last equality we have used the property \eqref{l:lcs} of the hermitian metric $g$. By Hopf maximum principle,  the kernel of $M^*$ consist of the constant functions. It follows (see \cite[App.~A]{lcs}) that the principal eigenvalues satisfy  $\lambda_0(M^*)= \lambda_0(M)=0$ so that $M(f)=0$ admits a unique up to scale  non-zero solution $f$ which never vanishes on $M$.  Letting
\begin{equation*}
\tilde \a := \frac{d_{-\a}f}{f}= \a + d \log |f|,
\end{equation*}
its is straightforward to check that \eqref{PSH} is equivalent to
$$dJ\tilde \a + \tilde \a \wedge J\tilde \a=0.$$
Noting that $[\tilde \a]=[\a] \neq 0$ (as $[\a]\in \mathcal{T}(S)$ by assumption, see \cite[Prop.~3.5]{lcs}), this concludes the proof of `(4)  $\Rightarrow$ (1)' in Case 4 too.  \hfill$\Box$
\begin{rem}\label{kahler-rank} The condition (1) of Theorem~\ref{thm:main} above gives a direct link to the theory of complex surfaces of K\"ahler rank one, see \cite{HL,CT}. Indeed, the condition (1) of Theorem~\ref{thm:main} implies that the $(1,1)$-form $\a\wedge J\a$ is closed and positive in the sense of \cite{CT}, thus  by \cite[Cor. 4.3]{CT}, forcing the K\"ahler rank of $S$ be equal to $1$. We shall use a further ramification of the theory  of \cite{CT} in the proof of Theorem~\ref{thm:A} below, see in particular Lemma~\ref{l:choise-toma}.
\end{rem}

\subsection{Proof of Theorem~\ref{thm:A}} By Theorem~\ref{thm:main}, we only need to show that if the condition (1) of Theorem~\ref{thm:main} holds true, then $S$ must be a blow-up of an Inoue--Bombieri surface.

As $b_1(S)=1$, the torsion-free part $H_1(S, \bb Z)^f$ of $H_1(S, \bb Z)$ is $\bb Z$, so that $S$ admits an infinite cyclic cover $\tilde S$ whose fundamental group is the kernel of the morphism $\pi_1(S) \to \big(\pi_1(S)/[\pi_1(S), \pi_1(S)]\big)^f$. Denote by $\gamma$ the deck transformation on $\tilde S$,  such that $S= \tilde S / \langle \gamma \rangle$.  We then have the following reinterpretation of condition (1) of Theorem~\ref{thm:main} in terms of the theory developed in \cite{CT}.
\begin{lem}\label{l:choise-toma} Let $S$ be a compact complex surface with $b_1(S)=1$ and  $\tilde S$ the infinite cyclic cover of $S$ with $S=\tilde S/\langle \gamma \rangle$. Then, the condition ${\rm (1)}$ of Theorem~\ref{thm:main} is equivalent to the existence of a positive pluriharmonic function $\tilde h$ on  $\tilde S$,  which is automorphic in the sense that $\tilde h \circ \gamma  = C \tilde h$ for a positive constant $C$.
\end{lem}
\begin{proof}  Suppose first that  $\a$ is a  closed $1$-form on $S$ satisfying the condition ${\rm (1)}$ of Theorem~\ref{thm:main}. Let $\tilde \a$ be the pull back to $\tilde S$ of $\a$. By the very construction of $\tilde S$, any closed $1$-form on $S$ pulls back to an exact $1$-form on $\tilde S$.   We thus can write $\tilde \a = d\tilde f$ for some smooth function $\tilde f$ on $\tilde S$. Using that $\tilde \a$ is invariant under the action of $\gamma$, it follows that $\tilde f$ satisfies
$$\tilde f \circ \gamma - \tilde f = c, $$
for some real constant $c$.  It is easily seen that the relation $d_{-\tilde \a} d^c_{-\tilde \a} (1)=0$ is equivalent to  $d d^c (e^{\tilde f}) =0$, i.e. $\tilde h:= e^{\tilde f}$ is a positive pluriharmonic function on $\tilde S$ satisfying  $\tilde h \circ \gamma = e^c \tilde h.$ 

Conversely, if $\tilde h$ is a positive pluriharmonic function on $\tilde S$ satisfying $\tilde h \circ \gamma = e^c \tilde h$, letting $\tilde \a: =  d \log \tilde h$ defines a closed,  $\gamma$-invariant $1$-form on $\tilde S$ which  satisfies $d_{-\tilde \a} d^c_{-\tilde \a} (1)=0$. It follows that $\tilde \a$ is the pull-back of a (non-zero)  closed  $1$-form $\a$ on $S$,  satisfying $d_{-\a}d^c_{-\a} (1)=0$.  As we noticed in the proof of Lemma~\ref{l:inoue}, such a form $\a$ defines a nontrivial class $[\a]\neq 0 \in H^1_{dR}(S, \bb R)$. \end{proof}
Theorem~\ref{thm:A} is thus a direct consequence of Theorem~\ref{thm:main} and \cite[Thm~1]{brunella-inoue}. \hfill$\Box$


\bigskip

Combining Theorem~\ref{thm:main} with some observations from \cite{lcs} leads to 
\begin{cor}\label{c:summary}  Let $S$ be a compact complex surface  with $b_1(S)=1$ and $S_0$ its minimal model. Then, 
\begin{enumerate}
\item[\rm (a)] If $S_0$ is a secondary Kodaira surface or a Hopf surface, then $\mathcal{T}(S)=(-\infty, 0)$ where $H^1_{dR}(S, \bb R)\cong \bb R$ is oriented by the sign of the degree with respect to any Gauduchon metric on $S$ of the flat holomorphic line bundle $\cal L_{a}$ associated to $a\in H^{1}_{dR}(S)$, see \cite{lcs}.
\item[\rm (b)] If  $S_0$ is an Inoue--Bombieri surface, then $\mathcal{T}(S)=\{a\}$.
\item[\rm (c)]  If  $S_0$ belongs to ${\rm VII}_0^+$, then $\mathcal{T}(S) \subset \bb R$ is a non-empty and open subset of $(-\infty, 0)$.
\end{enumerate}
\end{cor}
\begin{proof} (a) The case when $S_0$ is a Hopf surface is established in \cite[Prop.~5.1]{lcs}. The case when $S_0$ is a secondary Kodaira surface follows by a similar argument:
The pull-back by the blow-down map $b : S \to S_0$ defines an isomorphism  $b^*: H^{1}_{dR}(S_0, \bb R) \to H^1_{dR}(S, \bb R)$ (compare with Lemma~\ref{blow-up}(a)) which by \cite{tricerri,Vul} (see also \cite[Prop.~3.4]{lcs}) embeds $\mathcal{T}(S_0)$ into $\mathcal{T}(S)$. According to  \cite[Prop. 4.3]{lcs},  $\mathcal{T}(S_0) \subset \mathcal{T}(S) \subset (-\infty, 0)$, so it will be enough to show that 
$\mathcal{T}(S_0) =(-\infty, 0)$. As $S_0$ admits a Vaisman locally conformally K\"ahler metric by \cite{B}, the latter property follows from \cite[Lemma~3.7]{lcs} by noting that Vaisman metrics are pluricanonical.

The statement (b) is established in Lemma~\ref{l:inoue}.

(c) See Theorem~\ref{thm:main}.  \end{proof}

As another application of Theorem~\ref{thm:main}, one can  consider  {\it exact} LCS structures, i.e. LCS structures  for which $\omega = d_\a \eta = d\eta - \a \wedge \eta$ for some $1$-form $\eta$. Theorem~\ref{thm:EM} provides the existence of exact LCS structures with arbitrary large Lee classes  on any compact almost complex $2n$-manifold with non-zero $b_1(M)$. As an exact LCS structure does not admit symplectically embedded spheres (this is because by making a conformal modification of $\omega$ we can assume that $\a=0$ on a tubular neighbourhood of the sphere) they cannot tame the complex structure of a complex surface with a rational curve, in particular of a {\it non-minimal}  complex surface or a minimal surface in the class ${\rm VII}_0^+$ which contains a cycle of rational curves.  Similarly, as observed by  A. Otiman~\cite{otiman} 
\begin{cor}\label{c:exact}\cite{otiman} The Inoue--Bombieri surfaces admit no exact LCS structure taming its almost complex structure.
\end{cor}
\begin{proof} If $\omega=d_{\a} \eta$ is an exact LCS structure which tames $J$, so is then $\tilde \omega = d_{\tilde \a} \eta$ for each closed $1$-form $\tilde \a$ which is $\mathcal{C}^{\infty}$ close to $\a$. It follows that $\a$ is an interior point of $\mathcal{T}(S)$, a contradiction.\end{proof}
By contrast,  any secondary Kodaira surface and any Hopf surface admits a locally conformally K\"ahler structure with potential, i.e. of the form $\omega = d_{\a}d^c_{\a} f$   for some smooth function $f$ (this essentially follows from the fact that any  Vaisman locally conformally K\"ahler structure can be written up to scale as $\omega = d_{\a} d^c_{\a} (1)$, see \cite{Va} or \cite[(14)]{lcs}) so that on these minimal complex surfaces  any $[\a] \in \mathcal{T}(S)$  can be realized as the Lee class of an exact  LCS structure which tames the complex structure.  

\subsection{Proof of Theorem~\ref{thm:B}} Let $S$ be a compact complex surface with $b_1(S)=1$ and $S_0$ its minimal model. Identifying $H^1_{dR}(S_0, \bb R)\cong H^1_{dR}(S, \bb R)\cong \bb R$ via the blow-down map $b: S \to S_0$, we have  by \cite{Vul,tricerri},  $\mathcal{C}(S)= \mathcal{C}(S_0) \subset \bb R$. Thus, we can assume without loss that $S=S_0$ is minimal.  We consider the following three cases, see Theorem~\ref{thm:kodaira}.

\smallskip
\noindent
{\it Case 1}: $S \in {\rm VII}^+_0$. We shall use  a result of Goto~\cite[Thm.~2.3]{G} (compare with Proposition~\ref{prop:open} above).

\begin{thm}\cite{G}\label{thm:goto} Let $X$ be a compact complex manifold endowed with a locally conformally K\"ahler form $\omega$ with Lee form $\a$.  Suppose  that $H^3(X, L^*_{\a})=\{0\}$ and that for every $\bar \partial_{\a}$-closed  $(0,2)$-form $\psi$ there exists a $(0,1)$-form $\gamma$ such that $\partial_{\a}\psi= \partial_{\a}\bar \partial_{\a} \gamma$. Then, for any  closed $1$-form $\b$, there exists an $\varepsilon >0$ such that  for any $0<|t|<\varepsilon$, $X$ has a locally conformally K\"ahler form $\omega_t$ with Lee form $\a + t\b$.
\end{thm}
As  $S \in {\rm VII}^+_0$, we know by Lemma~\ref{l:vanishing}  that  $H^{0,2}_{\bar \partial_{\a}}(S, \bb C)\cong H^2(S, \cal L^*_{\a}) =\{0\}$,  which shows that  the second necessary condition of Theorem~\ref{thm:goto} for the class $[\a]\in \mathcal{C}(S)$ to be an interior point is aways satisfied. In other words, $[\a] \in \cal{C}(S)$ is an interior point for ${\cal C}(S)$ provided that $H^3_{d_{L^*}}(S, L^*) \cong H^1_{d_L}(S, L)=\{0\}$ where $L=L_{\a}$ (see \eqref{poincare}). By Theorem~\ref{thm:main}, the latter condition fails if and only if $\cal{C}(S) \subset \mathcal{T}(S)=\{[\a]\}$, showing that then ${\cal C}(S)$ is either empty or a single point.

\smallskip
\noindent
{\it Case 2}: $S$ is  a Hopf surface or a secondary Kodaira surface. We claim that $\mathcal{C}(S)=\mathcal{T}(S)=(-\infty, 0)\subset \mathbb{R}$. In fact, for Hopf surfaces  this is established in  \cite[Prop. 5.1]{lcs} whereas in the case when $S$ is a secondary Kodaira surface this follows from the arguments in the proof of Corollary~\ref{c:summary}(a). 

\smallskip
\noindent
{\it Case 3}: $S$ is an Inoue--Bombieri surface. In this case,  according to Lemma~\ref{l:inoue}, $\mathcal{C}(S) \subset \mathcal{T}=\{a\}$ is either empty or a point. Both case do appear, as noticed in \cite{B}.  \hfill$\Box$

\section{Vanishing results of twisted cohomologies for class ${\rm VII}$ surfaces.}
\subsection{Flat versus topologically trivial holomorphic line bundles}\label{s:complex-valued} Most of the theory developed in Section~\ref{s:preliminaries} for a {\it real} closed $1$-form $\a$ generalizes {\it mutatis mutandis} to the case when $\a$ is a closed complex-valued $1$-form: We can associate to such an $\a$  the deRham complex $d_{\a} : \mathcal{E}^k(X, \bb C) \mapsto \mathcal{E}^{k+1}(X, \bb C)$ with cohomology groups $H^k_{\a}(X, \bb C)$  and the  Dolbeault complexes with respect to the operator $\bar \partial_{\a} = \bar \partial - \a^{0,1} \wedge,$ which give rise to cohomology groups $H^{k,0}_{\bar \partial_{\a}}(X, \bb C)$.  Furthermore,  $\a$ induces a holomorphic structure on $\cal L_{\a}=\bb C \times X$, given by  $\bar\partial_{\a} s = (d s)^{0,1}  + \alpha^{0,1}\otimes s$ and  we have the identification
$$H^0(X, \Omega^k \otimes \cal L_{\a}) \cong H^{k,0}_{{\bar \partial}_{-\a}}(X, \bb C).$$ The conclusion of Lemma~\ref{blow-up}(b) holds true as well in this more general context.   

\smallskip
Recall that  equivalence classes of flat complex line bundles are classified by elements of $H^1(X, \bb C^*)$. By the short exact sequences
\begin{equation*}
\begin{split}
 \{0\} & \to \bb Z  {\hookrightarrow} \bb C   \stackrel{\exp 2\pi i \cdot}{\longrightarrow} \bb C^* \to \{1\}  \\
 \{0\} & \to \bb Z   {\hookrightarrow} \cal O \stackrel{\exp 2\pi i \cdot}{\longrightarrow}  {\cal O}^* \to \{1\}
 \end{split}
 \end{equation*}
we obtain the commutative digram
\begin{diagram}\label{commutative}
0  & \rTo  & H^1(X, \bb Z)  &  \rTo       & H^1(X, \bb C) & \rTo  & H^1(X, \bb C^{*}) & \rTo^{c_1} & H^2(X, \bb Z)  & \rTo \cdots  \\
             &          & \vEq                  &                 &   \dTo^j              &          &         \dTo^k                &         &            \vEq        &           \\                      
0  & \rTo  & H^1(X, \bb Z)  &  \rTo       & H^1(X, \cal O) & \rTo  & H^1(X, {\cal O}^{*}) & \rTo^{c_1} & H^2(X, \bb Z)  & \rTo \cdots  \\
\end{diagram}
The first line shows that for {any} {\it topologically trivial} flat holomorphic line bundle $\cal L \in H^1(X, \bb C^*)$, there exists  a closed complex-valued $\a$ with $\cal L =\cal L_{\a}$. Furthermore, if $j$ is injective, then $k$ is injective too whereas  if $j$ is an isomorphism, then $k$ is an isomorphism between the group $H_0^1(X, \bb C^*)$ of equivalent classes of topologically trivial flat holomorphic line bundles  and the group ${\rm Pic}^0(X)$ of equivalence classes of topologically trivial holomorphic line bundles.

In this section, we shall apply this construction in the special case of a compact complex surface $S$ with first Betti number $b_1(S)=1$. It is well-known (see e.g. \cite[(14)]{kodaira}) that on a compact complex surface $S$,  the morphism $j$ in the above diagram is always surjective and is an isomorphism iff $H^{1,0}(S, \bb C)=\{0\}$. As the latter property holds true for a complex surface with $b_1(S)=1$ (see \cite[Thm. 3]{kodaira}),  we have the following well-known (see e.g. \cite{LT})

\begin{lem}\label{l:kodaira} On a compact complex surface $S$  with $b_1(S)=1$,  $H^1_0(S, \bb C^*) \cong {\rm Pic}^0(S)$. In particular, for {\it any} holomorphic line bundle $\cal L \in {\rm Pic}^0(S)$ there exists a closed complex-valued $1$-form $\alpha$ such that $\cal L = \cal L_{\a}$. 
\end{lem}

\subsection{A characterization of Enoki surfaces}\label{s:enoki} An important class of examples of minimal complex surfaces  in the class ${\rm VII}_0^+$, called {\it Enoki surfaces}, was introduced and studied by Enoki in \cite{Enoki}. We shall use here the following characterization of such surfaces
\begin{thm}\label{thm:enoki}\cite{Enoki}. A minimal compact complex surface in  the class {\rm VII}$_0^+$ is an Enoki surface if and only if there exits a non-trivial holomorphic  line bundle $\cal L \in {\rm Pic}^0(S)$  which admits a meromorphic section $f$.  In this case, the divisor defined by $(f)$ is $m D$ for some $m\in \mathbb Z$, where $D$ is the unique cycle of rational curves of $S$.
\end{thm}
By virtue of Lemma~\ref{blow-up}(b) and the remarks at the beginning of Section~\ref{s:complex-valued}, it follows that 
\begin{cor}\label{c:enoki} A compact complex surface $S$ whose minimal model  is in class ${\rm VII}_0^+$  is a blow-up of an Enoki surface if and only if $S$ admits a non-trivial  holomorphic line bundle $\cal L \in {\rm Pic}^0(S)$ with $H^0(S, \cal L)\neq \{0\}$. 
\end{cor}
We recall the following general observation
\begin{lem} \label{dimformestordues} Let $S$ be a compact complex surface whose minimal model is in the class {\rm VII}$_0^+$. Then,  for any topologically trivial holomorphic  line bundle $\cal L$ 
 $${\rm dim}_{\bb C} H^0(S,\O^1\ot \cal L)\le 1.$$
 \end{lem} 
\begin{proof} By Lemma~\ref{blow-up}(b) and the remarks at the beginning of Section~\ref{s:complex-valued} (and since the fundamental group does not change under blow-down),  we can assume without loss of generality that $S=S_0$ is a {\it minimal} complex surface in the class ${\rm VII}^+_0$.  Let  $\beta_i\in H^0(S,\O^1\ot \cal L)$, $i=1,2,$ be two non-trivial  holomorphic 1-forms with values in $\cal L$.  It is clear (for instance by thinking of $\beta_i$ as smooth $(1,0)$-forms satisfying $\bar\partial_{-\a} \beta_i=0$, see \eqref{dolbeault-isom} and Section~\ref{s:complex-valued}, and Lemma~\ref{l:kodaira}) that $\beta_1\wedge \beta_2 \in  H^0(S, \O^2\otimes \cal L^2)= H^0(S,K_S\ot\cal L^2)$. By Lemma \ref{l:vanishing}, $\beta_1\w \beta_2 \equiv 0$. Letting $A \subset S$  be the vanishing locus of $\beta_1$, we thus have on $S\setminus A$,
 $$\beta_2=f\beta_1, $$ where $f$ is a holomorphic function defined on  $S\setminus  A$. We claim that $f$ extends as a meromorphic function over $A$, i.e. on $S$.

Let $D_{\rm max}$ be the {\it maximal divisor} of $S$ (see e.g. \cite{Nakamura}). As $A$ is an analytic subset of $S$, it is composed of curves contained in $D_{\rm max}$,  and of isolated points.  By Hartogs' extension  theorem, $f$ extends  holomorphically over the isolated points of $A$,  so we consider a point  $p\in A$  which belongs to an irreducible component $D_0$ of $D_{\rm max}$ with $D_0 \subset A$.  Let $U$ be an open neighbourhood of $p$ over which both vector bundles $\O^1$ and $\cal L$ trivialize. Since $\bb C\{z_1,z_2\}$ is a factorial ring,  we can write (with respect to  holomorphic coordinates $z=(z_1, z_2)$ on $U$)  $$\beta_i=\mu_i(z)(a_i(z)dz_1+b_i(z)dz_2),  i=1,2,$$ where $\mu_i(z), a_i(z), b_i(z)$ are holomorphic functions such that  the codimension of the vanishing locus  $Z(a_i,b_i)$ of  $a_i$ and $b_i$ is  $2$.  Thus, $Z(a_1,b_1)$ consists of isolated points in $U$. Avoiding theses points, at least one of the coefficients $a_1,b_1$ does not vanish at $p$, say $ a_1(p)\neq 0$. We then have (in a neighbourhood of $p$)
$$\beta_1\w \beta_2=\mu_1(z)\mu_2(z)\bigl(a_1(z)b_2(z)-a_2(z)b_1(z)\bigr)dz_1\w dz_2\equiv 0$$
hence $a_1b_2-a_2b_1\equiv 0$,  i.e. $\beta_2= \frac{\mu_2a_2}{\mu_1a_1}\beta_1$ where $\frac{\mu_2a_2}{\mu_1a_1}$ is a meromorphic function on $U$ which extends $f$ over $p$. It thus follows that $f$ extends meromorphically over $U$  minus the isolated points $Z(a_1,b_1)$, hence on $U$ (by  Levi's extension theorem). Thus, $f$ extends meromorphically on $S$. Since any meromorphic function on $S$ is constant (see \cite{kodaira}),   we conclude that ${\dim}_{\bb C}H^0(S,\O^1\ot\cal L)=1$. \end{proof}

Another feature of the Enoki surfaces is given by the following
\begin{lem}\label{enoki} Let $S$ be an Enoki surface. Then, for any non-trivial holomorphic line bundle $\cal L \in {\rm Pic}^0(S)$   $${\rm dim}_{\bb C} H^0(S, \O^1\otimes \cal L)={\rm dim}_{\bb C} H^0(S, \cal L)\le 1.$$ 
Moreover,   the equality ${\rm dim}_{\bb C} H^0(S, \O^1\otimes \cal L)={\rm dim}_{\bb C} H^0(S, \cal L)=1$  holds if and only if $\cal L = m[D], \ m \in \bb N^*,$ where $D$ is  the cycle of rational curves of $S$
\end{lem}
\begin{proof}  Since there is no non-trivial meromorphic functions, ${\rm dim}_{\bb C} H^0(S,\cal L)\le 1$ for any line bundle $\cal L\in {\rm Pic}(S)$. In \cite{Enoki}, the Enoki surfaces are obtained as compactifications of affine line bundles by a cycle $D=\sum_{i=0}^{n-1}C_i$  of $n=b_2(S)$ rational curves. Theorem~\ref{thm:enoki} then shows that ${\rm dim}_{\bb C} H^0(S, \cal L) =1$ if and only if $\cal L = m[D]$ with $m \in \bb N^*$.

By Lemma \ref{dimformestordues},  we also have ${\rm dim}_{\bb C} H^0(S,\O^1\ot \cal L)\le 1$ for any $\cal L\in {\rm Pic}^0(S)$. Enoki surfaces  can be  also described (see \cite[Thm.~1.19]{DK98}) by a polynomial germ of the form 
\begin{equation}\label{formenormale} F(z_1,z_2)=( z_1z_2^nt^n+\sum_{i=0}^{n-1} \a_it^{i+1}z_2^{i+1}, tz_2), \quad 0<|t|<1.
\end{equation}
In terms of \eqref{formenormale}, the maximal divisor $D$ has  local equation $z_2=0$. It follows from \eqref{formenormale}  that $\frac{dz_2}{z_2}$ is a meromorphic  $(1,0)$-form on $S$,  which has a pole of order $1$ along $D$, or equivalently, $\beta_0:= dz_2$ is a holomorphic $(1,0)$-form with values in $\cal L_0=[D]$ on $S$, showing that ${\rm dim}_{\bb C} H^0(S,\O^1\ot \cal L_0^m)={\rm dim}_{\bb C} H^0(S,\cal L_0^m)=1$ for any $m \in \bb N^*$. 

Let $\cal L \in {\rm Pic}^0(S)$ be such that ${\rm dim}_{\bb C} H^0(S,\O^1\ot \cal L)=1$,  and $\b\neq 0\in H^0(S,\O^1\ot \cal L)$. A similar argument as the one used in the proof of Lemma \ref{dimformestordues} shows that there is a meromorphic section $f$ of $\cal L\ot \cal L_0^{-1},$ such that $\b=f\b_0$. It thus follows from Theorem~\ref{thm:enoki} that $(f)=pD$, $p \in \bb Z$. Since $\b$ is a holomorphic form (and $\b_0$ does not  vanish  along $D$) we conclude that $p\in \bb N$, so that $\cal L=[pD]\ot [D]=[mD]$ with $m=p+1\in\bb N^\star$, i.e. ${\rm dim}_{\bb C} H^0(S, \cal L)=1$. \end{proof}

Our main objective, which will occupy the remainder of the section, is establishing the following partial converse of Lemma~\ref{enoki}, which characterizes Enoki surfaces by the existence of a special type of singular holomorphic foliation:
\begin{thm}\label{thm:Enoki-characterization} Let $S$ be a compact complex surface whose minimal model $S_0$ is in class ${\rm VII}_0^+$,  and whose fundamental group is isomorphic to $\mathbb Z$. Then, for any non-trivial holomorphic line bundle  $\cal L\in {\rm Pic}^0(S)$ $${\rm dim}_{\bb C} H^0(S,\O^1\ot\cal L)={\rm dim}_{\bb C} H^0(S,\cal L)\le 1$$ with ${\rm dim}_{\bb C} H^0(S,\O^1\ot\cal L)= 1$ if and only if $S_0$ is an Enoki surface.
\end{thm}

Recall that, by Lemma~\ref{l:kodaira}, on a class ${\rm VII}$ surface $S$ any $\cal L\in {\rm Pic}^0(S)$ can be written as  $\cal L = \cal L_{\a}$ for some closed complex-valued $1$-form $\a$. We then have

\begin{lem} \label{2isoms} Let $\cal L = \cal L_{\a}\in {\rm Pic}^0(S)$ be a non-trivial holomorphic line bundle on a complex surface $S$ with minimal model in the class ${\rm VII}^+_0$. Then, the following isomorphisms hold true.
\begin{itemize}
\item[\rm (a)] If $H^0(S,\cal L)=0$, then $H^0(S,\O^1\ot\cal L)\cong  H_{-\a}^1(S,\bb C).$
\item[\rm (b)]  If $H^0(S,\cal L)\neq 0$, then $H^0(S,\cal L)\stackrel{d}{\cong} H^0(S,\O^1\ot\cal L).$
\end{itemize}
\end{lem}
\begin{proof}   (a) In view of the identification \eqref{dolbeault-isom} (see also Section~\ref{s:complex-valued}),   we are going to construct  an isomorphism $s: H^{1,0}_{\bar \partial_{-\a}}(S, \bb C) \to H^1_{-\a}(S, \bb C)$.  Let  $\beta$ be a $(1,0)$-form  satisfying $\bar \partial_{-\a} \beta =0$.  Then, the $(2,0)$-form $\partial_{-\a} \beta $ satisfies $\bar\partial_{-\a} (\partial_{-\a} \beta)=0$ and,  therefore,  $\partial_{-\a} \beta =0$ since $H^{2,0}_{\bar \partial_{-\a}}(S, \bb C) \cong H^0(S, \O^2\otimes \cal L)=H^0(S, K_S\otimes \cal L)=\{0\}$ by Lemma~\ref{l:vanishing}. As $d_{-\a} \b = (\partial_{-\a} + \bar \partial_{-\a})(\b)=0$, we thus have  a natural  map
$$\begin{array}{cccc}
s:&H^{1,0}_{\bar\part_{-\a}}(S, \bb C)&\to& H^1_{-\a}(S,\bb C)\\
&&\\
&\b &\mapsto&[\b].
\end{array}$$
It is easy to see that $s$ is injective when $H^0(S, \cal L)=\{0\}$. Indeed, let $\b$ be a $\bar\partial_{-\a}$-closed $(1,0)$-form such that $[\b]=0$ in $H^1_{-\a}(S, \bb C)$. This means that there exists a complex-valued smooth  function $\varphi$ on $S$ with  $\b=d_{-\a} \varphi$. Considering bi-degree, it follows that  $\bar\part_{-\a} \varphi =0$, i.e.  $\varphi $ defines a section in $H^0_{\bar \partial_{-\a}}(S, \bb C) \cong H^0(S,\cal L)$, thus $\varphi =0$ and $\b=0$. 

To prove the surjectivity of $s$, let $\theta$ be a complex-valued $d_{-\a}$-closed $1$-form on $S$ with $[\theta] \neq 0 \in H^1_{-\a}(S, \bb C)$. Then the $(0,1)$-part $\theta^{0,1}$ of $\theta$ satisfies $\bar \partial_{-\a} \theta^{0,1} =0$, i.e. $\theta^{0,1}$ defines a class in $H_{\bar \partial_{-\a}}^{0,1}(S, \cal L) \cong H^1(S, \cal L)$. Using $H^0(S,\cal L)=\{0\}$,  Lemma~\ref{l:vanishing} and Riemann--Roch, we have $H^1(S, \cal L)=\{0\}$ which shows that $\theta^{0,1}= \bar \partial_{-\a} \varphi$ for some complex-valued smooth function $\varphi$. Thus, the $1$-form $\tilde \theta := \theta - d_{-\a} \varphi$ is another representative of $[\theta]$,  which is of type $(1,0)$. It thus follows that $\tilde \theta \in H^{1,0}_{\bar\part_{-\a}}(S, \bb C)$ and $s(\tilde \theta) = [\tilde \theta] = [\theta]$.

\smallskip
(b) Again, using the identifications \eqref{dolbeault-isom},  it is enough to  show that  the natural morphism  $$d_{-\a} : H^0_{\bar \partial_{-\a}}(S, \bb C) \to H^{1,0}_{\bar \partial_{-\a}}(S, \bb C)$$ is an isomorphism.  For any smooth complex-valued function $\varphi$ satisfying $\bar \partial_{-\a} \varphi =0$, the operator $d_{-\a}$ associates the $(1,0)$ form $\beta:=d_{-\a} \varphi$. As  $\beta$  is $d_{-\a}$-closed, it also satisfies $\bar \partial_{-\a} \beta =0$. The map $d_{-\a} : H^0_{\bar \partial_{-\a}}(S, \bb C) \to H^{1,0}_{\bar \partial_{-\a}}(S, \bb C)$ is injective as $H^0_{-\a}(S, \bb C) = \{0\}$ for $[\a] \neq 0 \in H^1_{dR}(S, \bb C)$, see \eqref{vanishing-0}.  It is surjective because  $H^{1,0}_{\bar \partial_{-\a}}(S, \bb C) \cong H^0(S, \O^1\otimes \cal L)$ must be  $1$-dimensional by Lemma~\ref{dimformestordues}. \end{proof}

\begin{rem}\label{sheaves} The isomorphisms in Lemma~\ref{2isoms} can be alternatively derived from the following exact sequences of sheaves 
\begin{equation}\label{suite1} 0\to \bb C(\cal L)\to \cal O(\cal L)\stackrel{d}{\to} d\cal O(\cal L)\to 0, \end{equation} 
\begin{equation}\label{suite2} 0\to d\cal O(\cal L)\to \O^1\ot\cal L\stackrel{d}{\to}\O^2\ot\cal L\to 0, \end{equation}
where,  for a flat holomorphic line bundle $\cal L \in H^1(S, \bb C^*)$,  $\bb C(\cal L)$ denotes  the sheaf of local  parallel sections of $\cal L,$ and $d$ is the deRham differential defined on smooth forms with values in $\cal L$ by using the flat connection on $\cal L$.   By the long exact sequence of cohomologies associated \eqref{suite2} and Lemma~\ref{l:vanishing} we deduce an isomorphism
$$i: H^0(S, d{\cal O}(\cal L))  \cong H^0(S, \O^1\otimes \cal L).$$
The two isomorphisms appearing in Lemma~\ref{2isoms} are then the natural maps $s= i\circ \delta$ and $d$ in the long cohomology sequence associated to \eqref{suite1}
\begin{equation*}\label{suite3}
\begin{diagram}
0 & \rTo & H^0(S, \bb C(\cal L))  & \rTo  & H^0(S, \cal L)  &  \rTo^d     & H^0(S, d\cal O(\cal L))       & \rTo^{\delta}  & H^1(S, \bb C(\cal L)) & \rTo    & H^1(S, \cal L)  & \rTo &           \\
    &          &                                    &             &                           &              &             \dEq^i                           &            &                                     &             &                            &        &            \\                      
    &          &                                    &            &                            &               &  H^0(S, \O^1\otimes \cal L) &             &                                     &             &                            &        &            
\end{diagram}
\end{equation*}
One then deduces (a)  from the vanishing of $H^0(S, \cal L)$ and $H^1(S, \cal L)$ (using Lemma~\ref{l:vanishing} and Riemann-Roch) and the fact that  if $\cal L= \cal L_{\a}$ for a closed complex-valued $1$-form $\a$,  the deRham-Weil theorem gives   
\begin{equation*}\label{complex-isom}
H^k(S, \bb C(\cal L)) \cong H^k_{-\a}(S, \bb C).
\end{equation*}
Similarly, using that $H^0(S, \bb C(\cal L))=\{0\}$  for a non-trivial  holomorphic line bundle $\cal L$  (as a parallel section of $\cal L$ is either identically zero or never vanishes),  (b) 
follows from the injectivity of $d$ and Lemma~\ref{dimformestordues}. \hfill$\Box$ \end{rem} 

\smallskip
\noindent
{\it Proof of Theorem~\ref{thm:Enoki-characterization}.} Again, using Lemma~\ref{blow-up}(b), we can assume without loss of generality that $S=S_0$ is a {\it minimal} complex surface, i.e. it is in the  class ${\rm VII}_0^+$.  As we have already observed, then $S$ does not admit non-constant meromorphic functions, thus for any  $\cal L \in {\rm Pic}^0(S)$ we have ${\rm dim}_{\bb C} \ H^0(S, \cal L) = 0,1$. 

\smallskip
\noindent
{\it Case 1}: There exists a non-trivial $\cal L \in {\rm Pic}^0(S)$,  such that  ${\rm dim}_{\bb C} \ H^0(S,\cal L)=1$. By Theorem~\ref{thm:enoki}, $S$ is  an Enoki surface so that Theorem~\ref{thm:Enoki-characterization} follows from Lemma~\ref{enoki}.

\smallskip
\noindent
{\it Case 2}:   For any non-trivial $\cal L \in {\rm Pic}^0(S)$, ${\rm dim}_{\bb C} H^0(S,\cal L)=0$.  We thus need to show that in this case $H^0(S, \O^1\otimes \cal L) =\{0\}$ for any non-trivial $\cal L \in {\rm Pic}^0(S)$. 

Suppose for contradiction that $H^0(S,\O^1\ot \cal L)\neq \{0\}$.  Let $\a$ be a closed {complex valued} $1$-form $\a$ such that  $\cal L = \cal L_{\a}$, so we have $H^0(S,\O^1\ot\cal L)\cong H^{1,0}_{\bar\part_{-\a}}(S, \bb C)$,  see Sect.~\ref{s:preliminaries}. Let $\b \neq 0$ be a $(1,0)$-form on $S$ such that $\bar \partial_{-\a} \b =0$. By Lemma~\ref{2isoms}, $\b$ satisfies
$d_{-\a} \b =0$. Let $p:\tilde S\to S$ be the universal covering space of $S$, $\tilde\a=p^*\a$, $\tilde\b=p^*\b$  and $\tilde \Phi :\tilde S\to \tilde S$ a biholomorphism such that $S= \tilde S/ \Gamma$ where $\Gamma \cong \bb Z$ is the infinite cyclic fundamental group of $S$ generated by $\tilde\Phi$. The forms $\tilde \a$ and $\tilde \b$ are invariant by the action of $\tilde \Phi$ as they are pull-backs of forms on $S$. As $\tilde \a$ is closed on $\tilde S$ (and $\tilde S$ is simply connected), there exists a complex-valued function $\tilde f$ on $\tilde S$,  such that $\tilde \a=d\tilde f$. The invariance of $\tilde \a$ under $\tilde \Phi$ then means $d(\tilde f\circ \tilde \Phi -\tilde f)=0$, so  there  exists  a constant $C\in\bb C$ such that 
\begin{equation}\label{1}
\tilde f\circ\tilde\Phi-\tilde f=C.
\end{equation}
Notice that the constant $C$ cannot be zero as then $f$ would be $\Gamma$-invariant and would descend to $S$ to define a primitive of $\a$ which contradicts the assumption that $\cal L = \cal L_{\a} \neq \cal O$.

As  $d(e^{\tilde f}\tilde\b)=e^{\tilde f}(d_{-\tilde\a}\tilde \b) =0$,  the $(1,0)$-form $e^{\tilde f}\tilde\b$ is closed  and therefore exact on $\tilde S$. Thus, there exists a complex-valued smooth  function $\tilde g$ on $\tilde S$,  such that $e^{\tilde f}\tilde\b=d\tilde g$. Considering bi-degree, $\bar\part \tilde g=0$ i.e. $\tilde g$ is holomorphic. Using that $\tilde\b$ is $\tilde \Phi$ invariant and \eqref{1}, it follows that
 $$\tilde \Phi ^* d\tilde g=e^{C}d\tilde g,$$
 i.e. there exists a constant $K\in \bb C$ such that
 \begin{equation*}\label{2}
 \tilde g\circ\tilde\Phi=e^{C}\tilde g+K. 
 \end{equation*}
Setting $\tilde h: =\tilde g+\frac{K}{e^{C}-1}$, we obtain a non-zero holomorphic function on $\tilde S$ satisfying $\tilde \Phi ^*\tilde h=e^{C}\tilde h$. Thus, $h:=e^{-\tilde f}\tilde h$ is a smooth complex-valued function on $\tilde S$ which is $\Gamma$-invariant and satisfies $\bar \partial_{-\tilde \a} h=0$. It follows that $h$ descends  to $S$ to define a non-zero section of  $H^0(S,\cal L)$, a contradiction.  \hfill $\Box$

\begin{rem}\label{cyclic-cover} There are no known examples of class ${\rm VII}_0^+$ surfaces whose fundamental group is not isomorphic to $\bb Z$. In general, as the first Betti number of any class ${\rm VII}^+_0$ surface $S$ is equal to $1$,   $S$ admits a unique infinite cyclic cover $\tilde S$. The arguments in the proof of Theorem~\ref{thm:Enoki-characterization} extend under the (a priori weaker) assumption $H^1(\tilde S, \bb R)=\{0\}$.  \hfill$\Box$
\end{rem}

\subsection{Proof of Theorem~\ref{thm:Z-vanishing}}

We recall  the following vanishing result obtained in \cite{leon}.
\begin{lem}\label{leon}\cite{leon} Let $S$ be a compact complex surface diffeomorphic to $(S^1\times S^3) \sharp\,n\overline{{\bb C} P^2}$, $n \in \bb N^*,$ and $L=L_\a, \ [\a]\neq 0 \in H^1_{dR}(S, \bb R)$ a non-trivial flat real line bundle. Then for $k\neq 2$, $H^k_{d_L}(S,L)=0$.
\end{lem}
\begin{proof} As $H^k_{d_L}(S, L)$ depend only upon the smooth structure of $S$, we can assume without loss that $S$ is a complex surface obtained by a diagonal Hopf surface $S_0$ by blowing up $n$ points. By \eqref{vanishing-0} and Lemma~\ref{blow-up}, it is enough to consider the case $n=0$, i.e. $S\cong S^1 \times S^3$.  It is well-known (see e.g. \cite{Va}) that this smooth manifold admits a complex structure and a compatible Vaisman product metric with a parallel Lee form $\a_0$. Applying \cite[Thm 4.5]{leon}, we conclude that the cohomology $H^k_{t\a_0}(S, \bb R)$ vanishes for each $t\neq 0$. As $b_1(S)=1$, it follows that  there exists $t\neq 0$ such that $H_{d_L}^k(S,L)\cong H^k_{t\a_0}(S, \bb R)=\{ 0\}$. \end{proof}

We then have
\begin{lem} \label{c:third} Let $S$ be a compact complex surface whose minimal model is in the class {\rm VII}$_0^+$. If there exists a non-trivial flat real line bundle $L=L_\a, [\a]\neq 0 \in H^1_{dR}(S, \bb R),$  such that $H_{d_L}^1(S, L)\neq \{0\}$,  then $H^0(S,\O^1\ot \cal L)\neq \{0\}.$
\end{lem}
\begin{proof}  Let $\cal L = L_\a \otimes \bb C$ be the corresponding flat holomorphic line bundle.  We first show that $H^0(S, \cal L) =\{0\}$. Indeed, if  $H^0(S, \cal L) \neq \{0\}$, then by Corollary~\ref{c:enoki},  $S$ must be obtained by blowing up an Enoki surface, and thus $S$ must be diffeomorphic to $(S^1\times S^3)\sharp n {\overline {\bb C P^2}}, n \in \bb N^*$ (see e.g. \cite{Nakamura2}). This contradicts  $H_{d_L}^1(S, L)\neq \{0\}$ (according to Lemma~\ref{leon}).

Thus, $H^0(S, \cal L)=\{0\}$ and by Lemma~\ref{2isoms}  and \eqref{deRham-isom} we have ${\rm dim}_{\bb R} H^1_{d_L}(S, L) = {\rm dim}_{\bb R} H^1_{-\a}(S, \bb R)= {\rm dim}_{\bb C} H^1_{-\a}(S, \bb C)={\rm dim}_{\bb C}  H^0(S, \O^1\otimes \cal L)$. \end{proof}

We can now prove Theorem~\ref{thm:Z-vanishing} (which in turn generalizes Lemma~\ref{leon}): As noticed in \cite{FP}, by \eqref{poincare},  \eqref{vanishing-0},  and \eqref{euler}, it is enough to show  $H^1_{d_{L}}(S, L) =\{0\}$. By Lemma~\ref{blow-up}, we can assume that $S$ is minimal whereas by Theorem~\ref{thm:kodaira}, Lemma~\ref{l:hopf},  and the fact that the fundamental group of Inoue--Bombieri surfaces is not isomorphic to $\bb Z$, we can also assume that $S$ is in the class ${\rm VII}^+_0$.

If  $H^1_{d_L}(S, L) \neq \{0\}$, by Lemma~\ref{c:third}  we will  have $H^0(S, \O^1\otimes \cal L) \neq \{0\}$ whereas Theorem~\ref{thm:Enoki-characterization} implies that $S$ must be an Enoki surface, and therefore $S$ must be diffeomorphic to $(S^1\times S^3) \sharp\,n\overline{{\bb C} P^2}$ (see e.g.~\cite{Nakamura2}). According to Lemma~\ref{leon},  this contradicts the assumption $H^1_{d_{L}}(S, L) \neq \{0\}$.  \hfill $\Box$

\end{document}